\documentclass[review]{elsarticle}

\usepackage{amssymb,amsmath,amsthm,amstext,amsfonts}
\usepackage{mathrsfs}
\usepackage{amsthm}
\usepackage{color}
\usepackage{functan}
\usepackage{graphicx}
\usepackage[latin1]{inputenc}
\usepackage{mathtools} 
\usepackage[pdftex]{hyperref}
\usepackage{framed,fancybox}
\usepackage{xcolor}
\usepackage{epsfig}
\usepackage{mathrsfs}
\usepackage[latin1]{inputenc}
\usepackage{accents}

\makeatletter \@addtoreset{equation}{section} \makeatother

\renewcommand\thetable{\thesection.\@arabic\c@table}

\newtheorem{theorem}{Theorem}[section]

\newtheorem{lemma}[theorem]{Lemma}
\newtheorem{proposition}[theorem]{Proposition}
\newtheorem{remark}[theorem]{Remark}

\newcommand{\R}{\mathbb{R}}
\newcommand{\N}{\mathbb{N}}
\newcommand{\K}{\mathbb{K}}
\newcommand{\eps}{\varepsilon}
\newcommand{\lraup}{\relbar\joinrel\rightharpoonup}
\newcommand{\bs}{\boldsymbol}

\begin{document}
	\title{Lagrange multipliers and characteristic functions}
	\author[um]{Davide Azevedo}
	\ead{davidemsa@math.uminho.pt}
	\author[um]{Lisa Santos\corref{cor1}}
	\ead{lisa@math.uminho.pt}
	\cortext[cor1]{Corresponding author}
	\address[um]{Centro de Matemática, Universidade do Minho, Campus de Gualtar, 4710-057 Braga, Portugal}


\begin{abstract}
We consider a stationary variational inequality with gradient constraint and obstacle. We prove that this problem can be described by an equation using a Lagrange multiplier and a characteristic function. The Lagrange multiplier contains information about the contact set of the modulus of the gradient of the solution with the gradient constraint, and the characteristic function is defined in the contact set of the solution with the obstacle. Moreover, given a convergent sequence of data, we prove the stability of the corresponding solutions.
\end{abstract}

\maketitle
\section{Introduction}

Variational inequalities can model different types of free boundary problems. The most well known one is the obstacle problem, which describes the deformation of a membrane over an obstacle attached on the boundary and subjected to a force. This problem has been intensively studied, appearing in many physical or biological models with different types of local or nonlocal operators. We refer to the study of the free boundary, the boundary of the coincidence set between the solution and the obstacle, indicating only the pioneer papers \cite{LewyStampacchia1969, Kinderlehrer1971, CaffarelliRiviere1976}. 

Another class of relevant variational inequalities are those with constraints on the derivatives, being the first papers on this type of problems related to the elastic-plastic torsion of a cilindrical bar \cite{Ting1966,LanchonDuvaut1967} and corresponding to variational inequalities with gradient constraint. Problems with constraints on the derivatives (for instance, on the curl, in the vector case)  model the behaviour of sandpiles, river networks or superconductors (see \cite{Prigozhin1994, Prigozhin1996_1, Prigozhin1996_2, RodriguesSantos2000, MirandaRodriguesSantos2009, MirandaRodriguesSantos2012,MirandaRodriguesSantos2020}). 

The two types of inequalities referred above share similar questions. For the sake of clarity, we define explicitly below the obstacle and the gradient constraint variational inequalities, in their simplest versions.

The obstacle problem consists on finding $u\in H^1_0(\Omega)$ such that
\begin{equation}\label{obsp}
u\in\K_\psi\qquad\int_\Omega\nabla u\cdot\nabla(v-u)\ge\int_\Omega f(v-u),\qquad\forall v\in\K_\psi,
\end{equation}
where $\K_\psi=\{v\in H^1_0(\Omega):v\ge\psi\}$, with $\psi_{|_{\partial\Omega}}\le 0$. Assuming sufficient regularity for $\psi$ and $f$, the existence of solution is immediate. It is possible to prove that \eqref{obsp} is equivalent to the complementary problem
$$\min\{-\Delta u-f, u-\psi\}=0\ \text{ in }\Omega,\quad u=0\ \text{ on }\partial\Omega.$$
This means that, when $u$ is strictly above the obstacle, the equation $-\Delta u=f$ is satisfied and in the coincidence set $\{u=\psi\}$ we only know that $-\Delta u\ge f$.
When $\psi$ is regular enough, we can write the complementary problem as
$$-\Delta u=f+(-\Delta\psi-f)^+\chi_{\{u=\psi\}},$$
where $\chi_A$ denotes the characteristic function of the set $A$. This characterization of the obstacle problem, by an equation involving the characteristic function of the contact set, can be extended to the two obstacles problem and even to the $N$-membranes or multiphase problems (see \cite{AzevedoRodriguesSantos2005,RodriguesSantos2009}). Concerning the $N$-membranes regularity of the free boundary (the boundaries of the coincidence sets of two, three, \ldots, $N$ membranes), we refer the papers \cite{LindgrenRazaniLindgren2009,SavinYu2019}.

The gradient constraint problem is the following: to find $u\in\K_g^\nabla$ such that
\begin{equation}\label{gradp}
	u\in\K_g^\nabla\qquad\int_\Omega\nabla u\cdot\nabla(v-u)\ge\int_\Omega f(v-u),\qquad\forall v\in\K_g^\nabla,
\end{equation}
where $\K_g^\nabla=\{v\in H^1_0(\Omega):|\nabla v|\le g\}$.
Here, the coincidence set is the set $\{|\nabla u|=g\}$ and its boundary is also called the free boundary. But, in this case, very little is known about it. Up to our knowledge, the only known result is mentioned in a paper of Figalli and Shahgholian (\cite{FigalliShahgholian2015}). They state the regularity of the free boundary of the gradient constraint problem when it is equivalent to a double obstacle problem. A sufficient condition for this equivalence is $\Delta g^2\le 0$ (see \cite{Santos2002}) and $f$ constant.

As in the obstacle problem, we can try to rewrite the variational inequality as a complementary problem. It is easy to verify that $-\Delta u=f$ when $|\nabla u|<g$, but we do not know what happens with the sign of $-\Delta u-f$ when $|\nabla u|=g$. In fact, the problem 
$$\max\{-\Delta u-f,|\nabla u|-g\}=0\quad\text{ in }\Omega,\quad u=0\ \text{ on }\partial\Omega$$
may be not equivalent to \eqref{gradp}. The condition $\Delta g^2\le0$ and $f$ constant is again sufficient to guarantee the equivalence between both problems (see \cite{Santos2002}).

However, in the general case, there exists another way of formulating problem \eqref{gradp} using an equation, which consists on finding a pair $(\lambda,u)$ in suitable spaces such that
\begin{equation}
	\label{ml}
	\begin{cases}
		-\nabla\cdot(\lambda\nabla u)=f\ \text{ in }\Omega,\quad u=0\ \text{ on }\partial\Omega,\\
		\lambda\ge 1,\quad |\nabla u|\le g,\quad (\lambda-1)(|\nabla u|-g)=0\ \text{ in }\Omega.
	\end{cases}
\end{equation}
It is easy to verify that the function $u$ solves the variational inequality \eqref{gradp} and so $u$ is unique. About the uniqueness of $\lambda$, the only known results  can be found in \cite{Brezis1972} and \cite{Santos1991}. Brèzis proved the uniqueness when $g\equiv 1$ and $f \equiv \beta$, where $\beta\in\R^+$, in \cite{Brezis1972} and Santos verified this result for the evolutionary case, with $g\equiv1$ and $f=f(t)$, in \cite{Santos1991}.

The function $\lambda$ contains information about the contact set $\{|\nabla u|=g\}$. When $|\nabla u|<g$ then $\lambda=1$ and $-\Delta u=f$, as in the case of the obstacle problem.
For many years, little was known about the existence of such a $\lambda$, except for the special case where $g\equiv 1$ and $f$ is either constant or depending only on $t$ in the evolutionary case (\cite{Brezis1972,Santos1991}). In these situations, $\lambda$ is an $L^\infty$ function. In 2005, in \cite{DePascaleEvansPratelli2005}, the result was extended, in the stationary case, still with $g\equiv1$, but with $f$ a non-constant $L^p$ function, to the degenerate case where $\lambda\ge0$ (corresponding to the transport densities problem), proving that $\lambda\in L^p(\Omega)$. The type of estimates needed to prove the result seem unlike to apply to different types of constraints, such as the Curl constraint, or to different types of operators.

 In 2013,  assuming that $g\in W^{2,\infty}(\Omega)$, $g\ge g_*>0$ and $f\in L^\infty(\Omega)$, the existence of solution of problem \eqref{ml}, with $\lambda\in L^\infty(\Omega)'$, the vector space of $\sigma$-finite measures, was proved in \cite{AzevedoMirandaSantos2013}. This result was improved in 2017 (\cite{AzevedoSantos2017}), establishing the existence of $\lambda\in L^p(\Omega)$, for any $p\in[1,\infty)$, whenever $g\in C^2(\overline \Omega)$ and $\Delta g^2\le 0$, exactly the same sufficient condition for the equivalence between the gradient constraint problem with a two obstacles problem. Note that, however, when $f$ is not constant, this equivalence is not proved. Furthermore, the existence of $\lambda_\delta\in L^\infty(\Omega)'$ was proved, for any positive $\delta$, replacing $1$ by $\delta$ in \eqref{ml}. Besides, it was proved that $\lambda_\delta$ converges weakly-* to $\lambda$ in $L^\infty(Q_T)'$, $\nabla u_\delta$ converges to $\nabla u$ in $L^p(\Omega)$, and the pair $(\lambda,u)$ solves  the transport densities problem, here with $f$ not necessarily constant.
 
 Several papers extended the above results to other types of operators, namely to the curl, the bi-Laplacian, the symmetric gradient or the fractional gradient, in the stationary and evolutionary cases (see \cite{RodriguesSantos2019, RodriguesSantos2019_2, Giuffre2020, AzevedoRodriguesSantos2020,AzevedoRodriguesSantos2023,AzevedoRodriguesSantos2024}).
 
A natural question that we can pose is what happens if we consider a problem with two different constraints, for example an obstacle and a gradient constraint. This is not a rhetorical question, because we can easily imagine a diffusion problem with a constraint on the modulus of a velocity field, also subject to the existence of one or more obstacles with whom the solution has contact. Besides, problems with obstacle and gradient constraints appeared naturally in \cite{AzevedoAzevedoSantos2023} constructing functions belonging to $\K_g^\nabla$ but that are below a given function $u$ of this convex set. We are now in a position to describe the main goal of this paper. We wish to study the possibility of writing this type of problems with an equation involving the characteristic function of the coincidence set with the obstacle and a Lagrange multiplier. To do so, here we work with assumptions that guarantee us that the Lagrange multiplier $\lambda$ is a function.
\vspace{2mm}

Along this paper, we assume that
\begin{equation}\label{omega}
	 \Omega\text{ is a bounded open subset of $\R^d, d\in\N$ with a $C^2$ boundary},
\end{equation}
\begin{equation}\label{feg}
	f \in L^\infty(\Omega),\quad g\in C^2(\overline\Omega)\quad g_*=\min g>0,\quad\Delta g^2\le0,
\end{equation}
\begin{equation}\label{psi}
\psi \in W^{2,\infty}(\Omega),\quad \psi_{|_{\partial\Omega}}\le0,\quad |\nabla\psi|<g.
\end{equation}

We define the convex set
\begin{equation*}
\K_{g,\psi}^\nabla = \big\{v\in H^1_0(\Omega): |\nabla v| \leq g, v\geq \psi, \text{a.e. in } \Omega \big\}
\end{equation*}
and we consider the variational inequality: find $u\in \K_{g,\psi}^\nabla$ such that
\begin{equation}\label{vi}
\int_\Omega \nabla u\cdot \nabla (v-u) \geq \int_\Omega f (v-u), \quad \forall v\in \K_{g,\psi}^\nabla.
\end{equation}

Consider also the problem: find $\lambda \in L^p(\Omega)$, where $p\in(1,\infty)$, and $u \in W^{1,\infty}_0(\Omega)$ such that
\begin{equation}\label{lmcf}
\begin{cases}
-\nabla \cdot \big(\lambda \nabla u \big) - \big(-\Delta \psi -f \big)^+ \chi_{ _{\{u=\psi\} }} =f \ \text{ in } \Omega, \\
|\nabla u| \leq g, \quad \lambda \geq 1, \quad (\lambda-1)(|\nabla u| - g) = 0 \quad \text{in} \ \Omega, \\
u\ge\psi  \quad \text{in} \ \Omega,\\
u=0  \quad \text{on} \ \partial \Omega.
\end{cases}
\end{equation}

The main result of this paper is the following:
\begin{theorem}\label{+}
	Assume \eqref{omega}, \eqref{feg} and \eqref{psi}. 
	Then, for $1\le p<\infty$, there exists
	$$(\lambda, u) \in L^p(\Omega) \times \big(W^{2,p}(\Omega)\cap W^{1,p}_0(\Omega)\big)$$
that solves \eqref{lmcf}. In addition, $u$ is the unique solution of the variational inequality \eqref{vi}.
\end{theorem}

We observe that we could have worked with the $q$-Laplace operator  but we decided to avoid technicalities that could make the reading less clear.

\vspace{2mm}

In Section 2, we define a family of approximating problems, proving existence of solution and obtaining {\em a priori} estimates for the solution which are independent of $\eps$.

Section 3 is dedicated to proving that problem \eqref{lmcf} has a solution $(\lambda,u)$, where $u$ is the unique solution of the variational inequality  \eqref{vi}. 

In Section 4, we study a continuous dependence result. For each $n\in \N$, let $(\lambda_n,u_n)$ be a certain solution of problem \eqref{lmcf}, with data $(f_n,g_n,\psi_n)$. We will prove the convergence of the sequence $(\lambda_n,u_n)_n$ to $(\lambda,u)$, solution of the same problem, but with data $(f,g,\psi)$, which is the limit, in suitable spaces, of the sequence $(f_n,g_n,\psi_n)_n$.

\section{Approximating problem}

In this section, we define a family of approximating problems that penalize and regularize the variational inequality \eqref{vi}. The solutions of these problems will be used to obtain the existence of solution of the problem \eqref{lmcf}.

Throughout the paper, we fix $r > \max\{2,d/2\}$ and assume $0 < \eps < 1$. Consider the increasing continuous functions $k_\eps, \theta_\eps:\R \rightarrow \R$ defined by
\begin{eqnarray}\label{ktheta}
k_\eps(s)= \begin{cases}
1& \text{if } s\leq 0\\
1+\big(\tfrac{s}{\eps}\big)^r& \text{if } s>0
\end{cases}
\quad \ \text{ and} & \quad
\theta_\eps(s)= \begin{cases}
-1& \text{if } s\leq -\eps\\
\frac{s}{\eps}& \text{if } \eps< s<0\\
0& \text{if } 0 \leq s.
\end{cases}
\end{eqnarray}
Define a family of approximating problem as follows: to find $u_\eps$ belonging to a suitable space, such that
\begin{equation}\label{approx}
\begin{cases}
-\nabla \cdot \big(k_\eps (| \nabla u_\eps|^2 -g^2) \nabla u_\eps \big) + \big(-\Delta \psi -f \big)^+ \theta_\eps (u_\eps - \psi) =f \ \text{ in } \Omega\vspace{2mm}\\
u_\eps = 0 \ \text{ on } \partial \Omega.
\end{cases}
\end{equation}
From now on, we set $\widehat k_\eps = k_\eps (| \nabla u_\eps|^2 -g^2)$. We will denote by $C$ different positive constants independent of $\eps$.

\begin{proposition}
Under the assumptions \eqref{omega}, \eqref{feg} and \eqref{psi}, problem \eqref{approx} has a unique solution $u_\eps \in W^{2,p}(\Omega)$, for any $1\le p<\infty$.
\end{proposition}
\begin{proof} The operator
	$$Lv=-\nabla\cdot\big(k_\eps(|\nabla v|^2-g^2)\nabla v\big)$$ 
	is quasilinear elliptic (for details see \cite{AzevedoMirandaSantos2013}) and the bounded operator $T:H^1_0(\Omega)\rightarrow H^{-1}(\Omega)$ defined by
\begin{equation*}
	\langle Tu,v \rangle = \int_\Omega \big(-\Delta \psi -f \big)^+ \theta_\eps (u - \psi) v, \quad \forall u,v \in H^1_0(\Omega)
	\end{equation*}
is monotone. In fact,
\begin{equation*}
\langle T(u-v),u-v \rangle=\int_\Omega \big(-\Delta \psi -f \big)^+ \big(\theta_\eps(u-\psi)-\theta_\eps(v-\psi)\big)(u-v)\ge0,
\end{equation*}
since $\theta_\eps$ is an increasing function. So, by well known results (see, for instance, \cite[Theorem 2.1]{Lions1969}), the existence and uniqueness of solution in $H^1_0(\Omega)$ follows. 

Setting 
\begin{equation}\label{feps}
	F_\eps=f-\big(-\Delta \psi -f \big)^+ \theta_\eps(u_\eps-\psi), 
\end{equation} 
 as $-1\le\theta_\eps\le0$, then $f\le F_\eps\le f+(-\Delta \psi-f)^+$ and so $F_\eps\in L^\infty(\Omega)$. As $u_\eps$ solves the problem 
\begin{equation}\label{versao}
	-\nabla\cdot(\widehat k_\eps\nabla u_\eps)=F_\eps\text{ in }\Omega,\quad u_\eps=0\text{ on }\partial\Omega,
	\end{equation}
then $u_\eps\in W^{2,p}(\Omega)$, for any $p\in[1,\infty)$ (however, notice that  $\|u_\eps\|_{W^{2,p}(\Omega)}$ may explode when $\eps\rightarrow0$).
\end{proof}

\begin{proposition} Under the assumptions \eqref{omega}, \eqref{feg} and \eqref{psi},
	the solution $u_\eps$ of problem \eqref{approx} satisfies
\begin{equation}\label{upsi}
	u_\eps\ge \psi-\eps.
	\end{equation}
	\end{proposition}
\begin{proof}
Note that $(\psi-u_\eps-\eps)^+>0$ if and only if $u_\eps-\psi<-\eps$ and then 
$$\theta_\eps(u_\eps-\psi)(\psi-u_\eps-\eps)^+=-(\psi-u_\eps-\eps)^+.$$
 So, multiplying the first equation of problem \eqref{approx} by $(\psi-u_\eps-\eps)^+$ and integrating over $\Omega$, we get
\begin{align*}
\int_\Omega\widehat k_\eps\nabla u_\eps\cdot\nabla(\psi-u_\eps-\eps)^+&=\int_\Omega f(\psi-u_\eps-\eps)^+ +\int_\Omega(-\Delta\psi-f)^+(\psi-u_\eps-\eps)^+\\
&\ge -\int_\Omega\Delta\psi(\psi-u_\eps-\eps)^+,
	\end{align*}
as $(-\Delta\psi-f)^+\ge -\Delta\psi-f$. Then
\begin{equation*}
	0\le\int_\Omega\widehat k_\eps\nabla u_\eps\cdot\nabla(\psi-u_\eps-\eps)^+-\int_\Omega\nabla\psi\cdot\nabla(\psi-u_\eps-\eps)^+
	=\int_\Omega\big(\widehat k_\eps\nabla u_\eps-\nabla\psi\big)\cdot\nabla(\psi-u_\eps-\eps)^+
\end{equation*}
and so, 
\begin{equation}\label{upsi2}
	\int_{\{\psi-u_\eps-\eps>0\}}\big(\widehat k_\eps\nabla u_\eps-\nabla\psi\big)\cdot\nabla(u_\eps-\psi)\le0.
\end{equation}
Now we prove that 
\begin{equation}\label{psiu}
	\big(\widehat k_\eps\nabla u_\eps-\nabla\psi\big)\cdot\nabla(u_\eps-\psi)\ge|\nabla(u_\eps-\psi)|^2,
	\end{equation}
which is equivalent to
$$(\widehat k_\eps-1)\big(|\nabla u_\eps|^2-\nabla u_\eps\cdot\nabla\psi\big)\ge0.$$
But, using the Cauchy-Schwartz inequality and \eqref{psi}, we have
$$(\widehat k_\eps-1)\big(|\nabla u_\eps|^2-\nabla u_\eps\cdot\nabla\psi\big)\ge (\widehat k_\eps-1)|\nabla u_\eps|(|\nabla u_\eps|-|\nabla\psi|) \ge (\widehat k_\eps-1)|\nabla u_\eps|(|\nabla u_\eps|-g),$$
recalling that $|\nabla\psi|\le g$. The right hand side is non-negative as, when $|\nabla u_\eps|\le g$, we have $\widehat k_\eps=1$ and, when $|\nabla u_\eps|>g$, we have $\widehat k_\eps>1$.
Then, using \eqref{upsi2} and \eqref{psiu}, we get
\begin{equation*}
\int_\Omega|\nabla(\psi-u_\eps-\eps)^+|^2=\int_{\{\psi-u_\eps-\eps>0\}}|\nabla(\psi-u_\eps)|^2\le 0,
\end{equation*}
concluding that $(\psi-u_\eps-\eps)^+\equiv0$, as it is zero on $\partial\Omega$. 
\end{proof}

We will now obtain {\em a priori} estimates independent of $\eps$.

\begin{proposition} Under the assumptions \eqref{omega}, \eqref{feg} and \eqref{psi},
the solution $u_\eps$ of problem \eqref{approx} satisfies the following a priori estimates:
\begin{equation}\label{firstestimates}
\| \widehat k_\eps \|_{L^1(\Omega)} \leq C, \quad \quad \| \nabla u_\eps \|_{L^{2r}(\Omega)} \leq C_r,
\end{equation}
where $C_r$ is a positive constant independent of $\eps$.
\end{proposition}

\begin{proof}
We start by proving the upper bound of $\| \widehat k_\eps \|_{L^1(\Omega)}$. Multiplying \eqref{approx} by $u_\eps$ and integrating, we obtain
$$\int_\Omega \widehat k_\eps | \nabla u_\eps|^2 + \int_\Omega \big(-\Delta \psi -f \big)^+ \theta_\eps (u_\eps - \psi)u_\eps = \int_\Omega f u_\eps.$$

As $-1\leq \theta_\eps \leq 0$, then
\begin{align*}
\int_\Omega \widehat k_\eps | \nabla u_\eps|^2 &\leq \int_\Omega fu_\eps + \int_\Omega\big(-\Delta \psi -f \big)^+|u_\eps| \\
&\leq \Big( \| f \|_{L^2(\Omega)} + \|\big(-\Delta \psi -f \big)^+ \|_{L^2(\Omega)}\Big) \| u_\eps \|_{L^2(\Omega)} \\
&\leq D \Big( \| f \|_{L^2(\Omega)} + \|\big(-\Delta \psi -f \big)^+\|_{L^2(\Omega)}\Big) \|\nabla u_\eps \|_{L^2(\Omega)},
\end{align*}
where $D$ is the Poincar\'e constant.

Since $\widehat k_\eps \geq 1$, we get, 
$$\left(\int_\Omega | \nabla u_\eps|^2\right)^\frac12 \leq D\big( \| f \|_{L^2(\Omega)} + \|\big(-\Delta \psi -f \big)^+\|_{L^2(\Omega)}).$$
Therefore
\begin{equation}\label{c1}
\int_\Omega \widehat k_\eps | \nabla u_\eps|^2 \leq D^2\big( \| f \|_{L^2(\Omega)} + \|\big(-\Delta \psi -f \big)^+\|_{L^2(\Omega)})^2.
\end{equation}
As
\begin{align*}
\int_\Omega \widehat k_\eps \leq \tfrac {1}{g^2_*} \int_\Omega g^2 \widehat k_\eps& = \tfrac {1}{g^2_*} \int_{\{|\nabla u_\eps|\le g\}}g^2  \widehat k_\eps +\tfrac {1}{g^2_*} \int_{\{|\nabla u_\eps|>g\}} g^2\widehat k_\eps\\
&\le \tfrac1{g_*^2}\int_\Omega g^2+\tfrac1{g_*^2}\int_\Omega \widehat k_\eps|\nabla u_\eps|^2,
\end{align*}
we conclude that
\begin{equation*}
\|\widehat k_\eps\|_{L^1(\Omega)} \leq \frac {1}{g^2_*} \big(\|g^2\|_{L^2(\Omega)} + C_1\big),
\end{equation*}
where $C_1$ is the constant of the right hand side of \eqref{c1}.

Now we will obtain the $\| \nabla u_\eps \|_{L^{2r}(\Omega)}$ estimate. We partially follow the calculations in \cite{AzevedoSantos2017}. Setting
\begin{equation}\label{aeps}
	A_{\eps} = \{x\in \R^d:|\nabla u_{\eps}|^2-g^2\ge\sqrt\eps\},
\end{equation}
we have, for $x\in A_\eps$,
\begin{align*}
	|\nabla u_\eps|^{2r} &\leq 2^{r-1} \big( |\nabla u_\eps|^2 - g^2 \big)^r + 2^{r-1} g^{2r} \\ 
	& \leq 2^{r-1} \eps^r\widehat k_\eps + 2^{r-1} g^{2r}.
\end{align*}
So
\begin{align*}
\int_\Omega |\nabla u_\eps|^{2r}&= \int_{\Omega\setminus A_\eps} |\nabla u_\eps|^{2r} + \int_{A_\eps} |\nabla u_\eps|^{2r} \\
& \leq \int_{\Omega\setminus A_\eps} (g^{2}+\sqrt\eps)^r + \int_{A_\eps} 2^{r-1} \big( \eps^r \widehat k_\eps + g^{2r} \big) \\
& \leq C\Big(\|g\|_{L^{2r}(\Omega)}^{2r} +  \|\widehat k_\eps\|_{L^1(\Omega)} +1\Big),
\end{align*}
with $C$ being a constant independent of $\eps$.
\end{proof}

	To obtain an {\em a priori} estimate of $\|\widehat k_\eps\|_{L^p(\Omega)}$, for any $p\in[1,\infty)$, we will adapt to our case the four results we will list below, which come from \cite{AzevedoSantos2017}, where the detailed proofs can be found. We note that it is enough to prove the result for  $p\ge2$. 
	
	The equation under analysis in \cite{AzevedoSantos2017} is $-\nabla(k_\eps(|\nabla u_\eps|^2-g^2)\nabla u_\eps)=f$ in $\Omega$, $u_\eps=0$ on $\partial\Omega$, and the assumptions on $f$ and $g$ are the same as ours. On the other hand, here we are working with the problem $-\nabla(k_\eps(|\nabla u_\eps|^2-g^2)\nabla u_\eps)=F_\eps$ in $\Omega$, $u_\eps=0$ on $\partial\Omega$, where $F_\eps$ is defined in \eqref{feps} and satisfies $\|F_\eps\|_{L^\infty(\Omega)}\le2\|f\|_{L^\infty(\Omega)}+\|\Delta\psi\|_{L^\infty(\Omega)}$, and so it is bounded independently of $\eps$.
	
	\begin{lemma}[see \cite{AzevedoSantos2017}, Lemma 2.3, p. 598]\label{lemma23}
		Assume \eqref{omega}, \eqref{feg} and \eqref{psi}, and let $u_\eps$ be the solution of problem \eqref{approx}. Then there exists a positive constant $C$ and $\eps_0>0$, depending only on $\|f\|_{L^\infty(\Omega)}$, $\|\Delta\psi\|_{L^\infty(\Omega)}$ and $\|g\|_{C^2(\overline\Omega)}$ such that, for any $0<\eps\le\eps_0$,
		\begin{equation*}
			\forall x\in\partial\Omega\quad|\nabla u_\eps(x)|^2\le g^2(x)+C\eps.
		\end{equation*}
	\end{lemma}
\begin{proof}
  The proof of Lemma 2.3 of \cite{AzevedoSantos2017} is based on the construction of  supersolutions and  subsolutions of the problem, which are zero in each chosen point $x_0\in\partial\Omega$. A careful analysis of that proof shows it can be used here, replacing $\|f\|_{L^\infty(\Omega)}$ by $2\|f\|_{L^\infty(\Omega)}+\|\Delta\psi\|_{L^\infty(\Omega)}$ along the proof.
\end{proof}
\begin{lemma}[\cite{AzevedoSantos2017}, Lemma 2.4, p. 599] \label{lemmaI} Assume \eqref{omega}, \eqref{feg} and \eqref{psi}, and let $u_\eps$ be the solution of problem \eqref{approx}. Set
	\begin{equation*}
		I_\eps=\tfrac{\partial u_\eps}{\partial\bs n}\Delta u_\eps-\tfrac{\partial u_\eps}{\partial x_i}\,\tfrac{\partial^2 u_\eps}{\partial x_i\partial\bs n},
	\end{equation*}
	where $\bs n$ is the outward unitary normal to $\partial\Omega$.
	Then $\|I_\eps\|_{L^\infty(\partial\Omega)}$
	is bounded independently of $\eps$.
\end{lemma}
\begin{proof}
	The proof uses a result of \cite{Ladyzhenskaya1969}, p. 20.
\end{proof}

\begin{lemma}[see \cite{AzevedoSantos2017}, Lemma 2.5, p. 599]
	Assume \eqref{omega}, \eqref{feg} and \eqref{psi}, and let $u_\eps$ be the solution of problem \eqref{approx}. Then, for any $2 \leq p <\infty$,
	\begin{equation}\label{secondestimate}
		\|\widehat k_\eps \|_{L^p(\Omega)} ^p\le C\int_\Omega\big(|f|^p+|\Delta\psi|^p+1\big)
		+D\left(\int_\Omega \widehat k_\eps^p\Delta g^2-\int_{\partial\Omega} \widehat k_\eps^p\tfrac{\partial g^2}{\partial\bs n}-2\int_{\partial\Omega}\widehat k_\eps^2I_\eps\right),
	\end{equation}
	where $C,D$ are positive constants independent of $\eps$.
	\end{lemma}
\begin{proof} We easily verify we can replace $f$, in the proof of Lemma 2.5 of \cite{AzevedoSantos2017}, by $F_\eps$. This means that in the first integral of \eqref{secondestimate}, we just need to replace $C\displaystyle\int_\Omega|f|^p$ by $C\displaystyle\int_\Omega|F_\eps|^p\le C_1\int_\Omega\big(|f|^p+|\Delta\psi|^p\big)$.
\end{proof}

\begin{lemma}[see \cite{AzevedoSantos2017}, Lemma 2.6, p. 601]\label{lemma27}
	Assume \eqref{omega}, \eqref{feg} and \eqref{psi}, and let $u_\eps$ be the solution of problem \eqref{approx}. Given $2\le p<\infty$, there exists $C_p>0$ and $\eps_0>0$ such that, for all $0<\eps<\eps_0$,  we have
	\begin{equation}\label{213}
 \|\widehat k_\eps\|_{L^p(\Omega)} \leq C_p,
\end{equation}
	where $C_p$ depends only on $\Omega$, $p$, $\|f\|_{L^\infty(\Omega)}$, $\|g\|_{C^2(\overline\Omega)}$ and $\|\Delta\psi\|_{L^\infty(\Omega)}$.
\end{lemma}
\begin{proof} 
We only need to look at the last three terms of \eqref{secondestimate}.

We have 
$$\int_\Omega \widehat k_\eps^p\Delta g^2\le0$$
because $\widehat k_\eps\ge0$ and $\Delta g^2\le0$, by assumption \eqref{feg}.
By Lemma \ref{lemma23},  for any $x\in\partial\Omega$, we know that $k_\eps(|\nabla u_\eps(x)|^2-g^2(x))\le 1+C^r$. Thus
$$-\int_{\partial\Omega} \widehat k_\eps^p\tfrac{\partial g^2}{\partial\bs n}\le (1+C^r)^p\int_{\partial\Omega}\left|\tfrac{\partial g^2}{\partial\bs n}\right|$$
is independent of $\eps$ and so is
$$-2\int_{\partial\Omega}\widehat k_\eps^2I_\eps\le 2(1+C^r)^2\|I_\eps\|_{L^\infty(\partial\Omega)}|\partial\Omega|,$$
by Lemma \ref{lemmaI}.
\end{proof}

\begin{proposition}
	Assume \eqref{omega}, \eqref{feg} and \eqref{psi}. Let $u_\eps$ be the solution of problem \eqref{approx}. Then
\begin{equation}\label{thirdestimate}
	 \|\widehat k_\eps \nabla u_\eps \|_{L^2(\Omega)} \leq C
\end{equation}
where $C$ is a positive constant independent of $\eps$. 
\end{proposition}
\begin{proof}
Just notice that
$$\int_\Omega \big(\widehat k_\eps \big)^2 |\nabla u_\eps|^2 \leq \left( \int_\Omega \big(\widehat k_\eps \big)^{2r'} \right)^\frac{1}{r'} \left( \int_\Omega |\nabla u_\eps|^{2r} \right)^\frac{1}{r}= \|\widehat k_\eps\|_{L^{2r'}(\Omega)}^2\|\nabla u_\eps\|_{\bs L^{2r}(\Omega)}^2,$$
and the right hand side is independent of $\eps$ by \eqref{firstestimates} and Lemma \ref{lemma27}.
\end{proof}

\section{The Lagrange multiplier--characteristic function problem}

In this section, we prove that problem \eqref{lmcf} has a solution $(\lambda,u)$ for each fixed $p\in[2,\infty)$ and that $u$  is the unique solution of the variational inequality \eqref{vi}.
Although the Lagrange multiplier will depend on $p$, we will denote it only by $\lambda$ as we are working with a fixed $p$.

We start with two auxiliary results, that are not stated here in their general setting.
\begin{lemma}[\cite{Rodrigues2005}, Lemma 2] \label{rod} 
Given $u,v\in W^{1,p}(\Omega)$, $1<p<\infty$, such that $\Delta u, \Delta v\in L^{dp'}(\Omega)$, where $p'$ is the conjugate exponent of $p$, we have
\begin{equation*}
	\Delta u=\Delta v\quad\text{ a.e. in }\{x\in\Omega: u(x)=v(x)\}.
	\end{equation*}
\end{lemma}

\begin{theorem}[\cite{Williams1981}, Theorem 1]\label{reg}
Suppose that \eqref{omega} is satisfied, $g\in C^2(\overline\Omega)$ and $f\in L^\infty(\Omega)$. Then, if $u$ is the solution of the variational inequality \eqref{gradp}, we have  $u\in W^{2,p}(\Omega)$, for any $1\le p<\infty$.
\end{theorem}

We also need the following auxiliary lemma.

\begin{lemma}\label{k_monotone}
The operator $\Phi_\eps:H^1_0(\Omega)\rightarrow H^{-1}(\Omega)$, defined as
$$\langle \Phi_\eps(u),v\rangle=\int_\Omega k_\eps(|\nabla u|^2 -g^2)\nabla u\cdot\nabla v,$$
is monotone.
\end{lemma}
\begin{proof}
It is enough to notice that, using the Cauchy-Schwarz inequality, 
\begin{multline*}
\big(k_\eps(|\nabla u|^2 -g^2)\nabla u-k_\eps(|\nabla u|^2 -g^2)\nabla v\big)\cdot\nabla(u-v)\\
\ge\big(k_\eps(|\nabla u|^2 -g^2)|\nabla u|-k_\eps(|\nabla u|^2 -g^2)|\nabla v|\big)(|\nabla u|-|\nabla v|)\ge 0,
\end{multline*}
since the real function $\varphi_\eps(s)=k_\eps(s^2-a)s$ satisfies $\varphi_\eps'\ge0$.
\end{proof}

Now we prove the main theorem of this paper.

\begin{proof} {\em of Theorem \ref{+}}	
	 Note that we only need to prove this result for $p$ large enough since, if $(\lambda_p,u)$ solves the problem \eqref{lmcf} for a given $p$, it also solves it for any $q<p$. 

We look to the approximating problem  \eqref{approx} written as in \eqref{versao}.
By the {\em a priori} estimates \eqref{firstestimates}, \eqref{213}, \eqref{thirdestimate} and as $-1 \leq \theta_\eps(u_\eps-\psi) \leq 0$, there exist  functions $u$, $\lambda$, $\bs \Lambda$, $\chi_{_*}$ and subsequences, with the same set of indexes, such that
 for any $2r\le p<\infty$,
\begin{align*}
&	\nabla u_\eps \underset{\eps \rightarrow 0}{\lraup} \nabla u \text{ in } \boldsymbol L^{2r}(\Omega)\text{-weak},\qquad&\theta_\eps(u_\eps-\psi)\underset{\eps\rightarrow0}{\lraup}\chi_{_*}\text{ in }L^p(\Omega)\text{-weak},\\
&\widehat k_\eps \underset{\eps \rightarrow 0}{\lraup} \lambda \text{ in } L^p(\Omega)\text{-weak},\qquad&\widehat k_\eps \nabla u_\eps \underset{\eps \rightarrow 0}{\lraup} \boldsymbol \Lambda \text{ in } \boldsymbol L^2(\Omega)\text{-weak}.
\end{align*}
Recalling that $d<2r\le p$, by the compact Sobolev inclusion of $W^{1,p}_0(\Omega)$ into $C^{0,1-d/p}(\overline\Omega)$, we have $u_\eps\underset{\eps\rightarrow 0}\longrightarrow u$ in $C^{0,1-d/p}(\overline\Omega)$.

We split the proof into several steps.

\vspace{2mm}

{\bf Step 1:} $u\in W^{2,p}(\Omega)$

We first prove that $u\in\K_{g}$. To verify that $|\nabla u|\le g$, recall the definition of $A_\eps$ in \eqref{aeps} and that $k_\eps$ is an increasing function. In this set, we have $\widehat k_\eps=k_\eps(|\nabla u_\eps|^2-g^2)\ge k_\eps(\sqrt\eps)=1+\frac1{\eps^{\frac{r}2}}$ and we observe that
$$\frac1{\eps^{\frac{r}2}}|A_\eps|=\int_{A_\eps}\frac1{\eps^{\frac{r}2}}\le\int_{A_\eps}\widehat k_\eps\le\int_\Omega \widehat k_\eps\le C$$
and then
\begin{align*}
	\int_\Omega \big(|\nabla u|^2-g^2\big)^+&\le\varliminf_{\eps\rightarrow0}\int_\Omega\big(|\nabla u_\eps|^2-g^2\big)^+\\
	&=\varliminf_{\eps\rightarrow0}\Big(\int_{\Omega\setminus A_\eps}\big(|\nabla u_\eps^2|-g^2\big)^+ +\int_{A_\eps}\big(|\nabla u_\eps|^2-g^2\big)^+\Big)\\
	&\le \varliminf_{\eps\rightarrow0}\Big(\sqrt{\eps}|\Omega|+\Big(\|(|\nabla u_\eps|^2-g^2)^+\|_{L^2(\Omega)}|A_\eps|^\frac12\Big)\\
	&\le \varliminf_{\eps\rightarrow0}\Big(\sqrt{\eps}|\Omega|+\big(\|\nabla u_\eps\|^2_{\bs L^4(\Omega)}+\|g\|_{L^4(\Omega)}^2\big)\,C^\frac12\eps^\frac{r}4\Big)=0.
\end{align*}
So, $|\nabla u|\le g$ a.e..

Now we verify that $u$ satisfies the variational inequality \eqref{gradp} with $f$ replaced by $F=f-(-\Delta\psi-f)^+\chi_{_*}$. To do so, multiply the equation \eqref{approx} by $w-u_\eps$, with $w\in\K_g^\nabla$.
As $k_\eps(|\nabla w|^2-g^2)\equiv1$, we have
\begin{equation*}
	\int_\Omega\big(\widehat k_\eps\nabla u_\eps-\nabla w\big)\cdot\nabla(w-u_\eps)+\int_\Omega\nabla w\cdot\nabla(w-u_\eps)=\int_\Omega F_\eps(w-u_\eps),
\end{equation*}
where $F_\eps$ is defined in \eqref{feps}. Using Lemma \ref{k_monotone}, we obtain
$$\int_\Omega\nabla w\cdot\nabla(w-u_\eps)\ge\int_\Omega F_\eps(w-u_\eps).$$
Since $\nabla u_\eps\underset{\eps\rightarrow0}\lraup\nabla u$ in $L^2(\Omega)$-weak, $u_\eps\underset{\eps\rightarrow0}\longrightarrow u$ in $C(\overline\Omega)$ and $F_\eps\underset{\eps\rightarrow0}\lraup F$ in $L^2(\Omega)$-weak, we get
$$\int_\Omega\nabla w\cdot\nabla(w-u)\ge\int_\Omega F(w-u),\quad\forall w\in\K_g^\nabla.$$

Using the argument in Minty's Lemma, taking, for any $w\in\K_g^\nabla$ the test function $v=u+\mu(w-u)\in\K_g^\nabla$, where $\mu\in(0,1]$, we obtain
$$\mu\int_\Omega\nabla(u+\mu(w-u))\cdot\nabla(w-u)\ge\mu\int_\Omega F(w-u).$$
Dividing both members by $\mu$ and letting $\mu\rightarrow0^+$, we prove that $u$ solves the variational inequality \eqref{gradp} with $f$ replaced by $F$.  Finally, as $F\in L^\infty(\Omega)$, the regularity of $u$ is obtained applying Theorem \ref{reg}.

\vspace{2mm}

{\bf Step 2:} $u$ is the unique solution of the variational inequality \eqref{vi}

We already proved in the previous step that $|\nabla u|\le g$. As $u_\eps\ge\psi-\eps$ by \eqref{upsi}, we get $u\ge\psi$, and so $u\in\K_{g,\psi}^\nabla$. Now we multiply the equation \eqref{approx} by $v-u_\eps$, where $v\in\K_{g,\psi}^\nabla$. Note that $\theta_\eps(v-\psi)\equiv0$ and 
\begin{equation*}
	\int_\Omega(-\Delta\psi-f)^+\theta_\eps(u_\eps-\psi)(v-u_\eps)=	\int_\Omega(-\Delta\psi-f)^+\big(\theta_\eps(u_\eps-\psi)-\theta_\eps(v-\psi)\big)\big((v-\psi)-(u_\eps-\psi)\big)\le0
\end{equation*}
because the function $\theta_\eps$ is increasing. Using this and Lemma \ref{k_monotone}, we get
\begin{align*}
\int_\Omega f(v-u_\eps)&=\int_\Omega k_\eps(|\nabla u_\eps|^2-g^2)\nabla u_\eps
+\int_\Omega(-\Delta\psi-f)^+\theta_\eps(u_\eps-\psi)(v-u_\eps)\\
&\le\int_\Omega\nabla v\cdot\nabla(v-u_\eps) ,\quad\forall v\in\K_{g,\psi}^\nabla.
\end{align*}
Letting first $\eps\rightarrow0$ and using again the argument as in Minty's Lemma, the conclusion follows.
\vspace{2mm}

{\bf Step 3:} $(-\Delta\psi-f)^+\chi_{_*}=-(-\Delta\psi-f)^+\chi_{\{u=\psi\}}$

Observe first that $\theta_\eps(u_\eps-\psi)(u_\eps-\psi)^+\equiv0$ and so $\chi_{_*}(u-\psi)^+\equiv0$, concluding that $\chi_{_*}=0$ in $\{u>\psi\}$.

As $u\in W^{2,p}(\Omega)\subset C^{1,1-d/p}(\overline\Omega)$ if $p>d$, then $|\nabla u|$ is continuous and the set $\mathcal U=\{|\nabla u|<g\}$ is open. So, given $\varphi\in \mathcal D(\mathcal U)$, for $\delta$ sufficiently small, we have $v=u\pm\delta\varphi\in\K_g^\nabla$. Thus, we can use it as test function in \eqref{gradp} with $f$ replaced by $F$, obtaining, after integration by parts,
$$\pm\delta\int_\Omega(-\Delta u-F)\varphi\ge0,\quad\forall\varphi\in\mathcal D(\mathcal U).$$
Therefore
$$-\Delta u-f=-(-\Delta\psi-f)^+\chi_{_*}\quad\text{ in }\mathcal U$$
and, in particular, as $-1\leq \chi_{_*} \leq 0$, then $-(-\Delta\psi-f)^+\chi_{_*}\ge 0$ and also $-\Delta u-f\ge0$.

Observe now that, as $\nabla u=\nabla\psi$ in $\{u=\psi\}$ and, by assumption \eqref{psi}, $|\nabla\psi|<g$, then $\{u=\psi\}\subset\mathcal U$. Using Lemma \ref{rod}, we know that $\Delta u=\Delta\psi$ in $\{u=\psi\}$. So, in the set $\{u=\psi\}$, we have
$$-(-\Delta \psi-f)^+\chi_{_*}=-\Delta u-f=(-\Delta u-f)^+=(-\Delta\psi-f)^+=(-\Delta\psi-f)^+\chi_{\{u=\psi\}}$$
and the conclusion follows.

\vspace{2mm}

{\bf Step 4:} Problem \eqref{lmcf} has a solution

From \eqref{versao}, we get
\begin{equation}\label{uepsvar}
	\int_\Omega\widehat k_\eps\nabla u_\eps\cdot\nabla\varphi=\int_\Omega F_\eps\varphi,\qquad\forall \varphi\in H^1_0(\Omega)
	\end{equation}
and, passing to the limit when $\eps\rightarrow0$, , we obtain
\begin{equation}\label{LimLambda}
	\int_\Omega\bs\Lambda\cdot\nabla\varphi=\int_\Omega F\varphi,\qquad\forall\varphi\in H^1_0(\Omega).
\end{equation}

Now we will prove that 
\begin{equation}\label{lambdaLambda}
	\int_\Omega\lambda|\nabla u|^2=\int_\Omega\bs\Lambda\cdot\nabla u,
	\end{equation}
 following the steps in \cite{AzevedoSantos2017}, which are presented here for the sake of completeness. Observe that $\lambda\in L^p(\Omega)$ but $\nabla u\in L^\infty(\Omega)$, as $|\nabla u|\le g^*$. Note first that, as $\widehat k_\eps\ge1$, then $\lambda\ge 1$.
Using $u_\eps$ as test function in \eqref{uepsvar}, letting $\eps\rightarrow0$,  we get
\begin{equation}\label{kepsLambda}
\lim_{\eps\rightarrow0}	\int_\Omega\widehat k_\eps|\nabla u_\eps|^2=\lim_{\eps\rightarrow0}\int_\Omega F_\eps u_\eps=\int_\Omega Fu=\int_\Omega\bs\Lambda\cdot\nabla u,
\end{equation}
since $u_\eps$ converges strongly to $u$ in $L^2(\Omega)$ and applying \eqref{LimLambda}.
 
Observing that 	$(\widehat k_\eps-1)(|\nabla u_\eps|^2-g^2)\ge0$,  then
\begin{equation}\label{maior}
\widehat k_\eps|\nabla u_\eps|^2\ge \widehat k_\eps g^2+\big(|\nabla u_\eps|^2-g^2\big).
	\end{equation}
Besides, as $\lambda\ge1$ and $|\nabla u|\le g$, then
\begin{equation}\label{outra}
(\lambda-1)(g^2-|\nabla u|^2)\ge0.
\end{equation} 
So, using \eqref{kepsLambda} and \eqref{maior},
 \begin{align}\label{estxx'}
\nonumber\int_\Omega\bs\Lambda\cdot\nabla u&=\lim_{\eps\rightarrow0}\int_\Omega\widehat k_\eps|\nabla u_\eps|^2\ge\int_\Omega\lambda g^2+\varliminf_{\eps\rightarrow0}\int_\Omega\big(|\nabla u_\eps|^2-g^2\big)\\
&\ge\int_\Omega\lambda g^2+\int_\Omega\big(|\nabla u|^2-g^2\big)=\int_\Omega\big(\lambda-1\big)\big(g^2-|\nabla u|^2\big)+\int_\Omega\lambda|\nabla u|^2\\
\nonumber&\ge\int_\Omega\lambda|\nabla u|^2,
\end{align}
with the last inequality being verified by using \eqref{outra}.

We also have
\begin{align*}
0\le&\lim_{\eps\rightarrow0}\int_\Omega\widehat k_\eps|\nabla (u_\eps-u)|^2=\lim_{\eps\rightarrow0}\left(\int_\Omega\widehat k_\eps|\nabla u_\eps|^2-2\int_\Omega\widehat k_\eps\nabla u_\eps\cdot\nabla u+\int_\Omega\widehat k_\eps|\nabla u|^2\right)\\
	=&\int_\Omega\bs\Lambda\cdot\nabla u-2\int_\Omega\bs\Lambda\cdot\nabla u+\int_\Omega\lambda|\nabla u|^2\\
	=&-\int_\Omega\bs\Lambda\cdot\nabla u+\int_\Omega\lambda|\nabla u|^2,
\end{align*}
concluding, together with \eqref{estxx'}, that \eqref{lambdaLambda} holds. In particular,  as $\widehat k_\eps\ge1$,
\begin{equation}\label{strong}
\lim_{\eps\rightarrow0}	\int_\Omega|\nabla(u_\eps-u)|^2\le \lim_{\eps\rightarrow0}\int_\Omega\widehat k_\eps|\nabla( u_\eps-u)|^2=0,
	\end{equation}
proving the strong convergence of $u_\eps$ to $u$ in $H^1_0(\Omega)$.
Observe that, for any $\bs\xi\in\bs{\mathcal D}(\Omega)$,
$$\left|\int_\Omega\big(\bs\Lambda-\lambda\nabla u)\cdot\bs\xi\right|=\lim_{\eps\rightarrow0}\left|\int_\Omega\widehat k_\eps\nabla (u_\eps-u)\cdot\bs\xi\right| \le\lim_{\eps\rightarrow0} \Big(\int_{\Omega}\widehat k_\eps|\nabla(u_\eps-u)|^2\Big)^\frac12\Big(\int_\Omega\widehat k_\eps\Big)^\frac12\|\bs\xi\|_{\bs L^\infty(\Omega)}=0$$
and, by density,
$\bs\Lambda=\lambda\nabla u$ in $\bs L^2(\Omega).$ So, using \eqref{LimLambda},
\begin{equation*}
	\int_\Omega\lambda\nabla u\cdot\nabla\varphi=\int_\Omega F\varphi,\qquad \forall\varphi\in H^1_0(\Omega).
\end{equation*}

Using the inequality \eqref{maior}, we get $(\widehat k_\eps-1)(|\nabla u_\eps|-g)\ge0$, and, from the strong convergence of $u_\eps$ to $u$ in $H^1_0 (\Omega)$ proved in \eqref{strong}, we conclude that $(\lambda-1)(|\nabla u|-g)\geq 0$. Together with \eqref{outra}, this implies that  $(\lambda-1)(|\nabla u|-g)= 0$.

\end{proof}

\begin{remark}\label{lmcf_vi}
We saw in the last proof that the solution of the variational inequality \eqref{vi} can be obtained as the limit of a subsequence of solutions of the approximating problems \eqref{approx}. We will now see that, if $(\lambda,u)$ is any solution of the Lagrange multiplier-characteristic function \eqref{lmcf}, then $u$ solves the variational inequality \eqref{vi}. In fact, as
$(\lambda-1)(|\nabla u|-g)=0$, for any $v\in\K_{g,\psi}^\nabla$, we have
$$(\lambda-1)\nabla u\cdot\nabla(v-u)\le(\lambda-1)(|\nabla u|\,|\nabla v|-|\nabla u|^2)=(\lambda-1)|\nabla u|(g-|\nabla u|)=0$$
and
$$-(-\Delta\psi-f)^+\chi_{\{u=\psi\}}(v-u)=-(-\Delta\psi-f)^+\chi_{\{u=\psi\}}(v-\psi)\le0.$$
So, multiplying the equation \eqref{lmcf} by $v-u$ and integrating, we get
$$\int_\Omega\lambda\nabla u\cdot\nabla(v-u)-\int_\Omega(-\Delta\psi-f)^+\chi_{\{u=\psi\}}(v-u)=\int_\Omega f(v-u)$$
and, using the two inequalities above, we obtain that
$$\int_\Omega\nabla u\cdot\nabla(v-u)\ge\int_\Omega f(v-u),$$
where $v$ is any function in $\K_{g,\psi}^\nabla$ and $u\in\K_{g,\psi}^\nabla$, as $u$ solves \eqref{lmcf}.
\end{remark}

\begin{remark} $ $
	\begin{enumerate}
		\item A relevant question is what happens if we weaken the assumption $|\nabla\psi|< g$ by assuming only $|\nabla\psi|\le g$. In this situation, the sets $\{u=\psi\}$ and $\{|\nabla u|=g\}$ may intersect and the identification of $(-\Delta\psi-f)^+\chi_{_*}$ with $-(-\Delta\psi-f)^+\chi_{\{u=\psi\}}$ becomes more difficult or may not even be possible.
	\item  We could also think about what happens if there is no restriction on $|\nabla\psi|$, which implies that the contact of $u$ and $\psi$ can only happen in the region where $|\nabla\psi|\le g$.
	\end{enumerate}
	\end{remark}
	
	\begin{remark}
		We observe that the solution of the variational inequality with gradient constraint is also solution of a variational inequality with obstacle and gradient constraint.
		
		Let $u$ be the unique solution of the variational inequality \eqref{gradp}. Then, setting $g^*=\|g\|_{L^\infty(\Omega)}$, we have
		$$\forall x,y\in\overline\Omega\quad|u(x)-u(y)|\le g^*|x-y|.$$
		Then, as $u=0$ on $\partial\Omega$, we conclude that $-g^*\text{diam}(\Omega)\le u(x)\le g^*\text{diam}(\Omega)$, where $\text{diam}(\Omega)$ is the diameter of the set $\Omega$. Therefore, as $u\in \K_{g,-g^*\text{diam}(\Omega)}\subseteq \K_g^\nabla$, then $u$ is the solution of the variational inequality with gradient constraint $g$ and obstacle $-g^*\text{diam}(\Omega)$.
	\end{remark}
	
	\begin{remark}
		The solution of variational inequality with obstacle is also solution of a variational inequality with obstacle and gradient constraint.
		
		Let $u$ be the solution of the variational inequality \eqref{obsp}. Considering the approximating problem
		$$-\Delta u_\eps+(-\Delta\psi-f)^+\theta_\eps(u_\eps-\psi)=f\quad\text{ in }\Omega,\qquad u_\eps=0\quad\text{ on }\partial\Omega,$$
		where $\theta_\eps$ is defined in \eqref{ktheta}, set $G_\eps=f-(-\Delta\psi-f)^+\theta_\eps(u_\eps-\psi)$ and note that $u_\eps$ solves the problem $-\Delta u_\eps=G_\eps$ in $\Omega$, $u_\eps=0$ on $\partial\Omega$. As
		 $$\|G_\eps\|_{L^\infty(\Omega)}\le 2\|f\|_{L^\infty(\Omega)}+\|\Delta\psi\|_{L^\infty(\Omega)},$$
		 then, for any $1\le p<\infty$, $\|u_\eps\|_{W^{2,p}(\Omega)}\le C$ and
		then $u_\eps\underset n\lraup u$ in $W^{2,p}(\Omega)$-weak. So
		$$\|u\|_{W^{2,p}(\Omega)}\le\varliminf_{\eps\rightarrow0}\|u_\eps\|_{W^{2,p}(\Omega)}\le C.$$
	 Thus, as $u\in C^{1,1-d/p}(\overline\Omega)$ if $p>d$, there exists $g^*\in\R^+$ such that $|\nabla u|\le g^*$. Then, as $u\in\K_{g^*,\psi}\subseteq \K_\psi$, $u$ solves the variational inequality with gradient constraint $g^*$ and obstacle $\psi$.
	\end{remark}

\section{Continuous dependence}

In this section, we study the behaviour of certain sequences of solutions $(\lambda_n,u_n)_n$ of problem \eqref{lmcf}, when the given data converges, when $n\rightarrow\infty$, in suitable spaces. 

\begin{theorem}\label{dependence} Suppose that \eqref{omega} is satisfied and, for each $n\in\N$, $(f_n,g_n,\psi_n)$ satisfies the assumptions \eqref{feg}  and \eqref{psi}. Suppose, in addition, that there exists $(f,g,\psi) \in L^\infty(\Omega) \times C^2(\overline\Omega) \times W^{2,\infty}(\Omega)$ such that, for a subsequence,
	\begin{equation}\label{limn}
			f_n\underset n\longrightarrow f\quad\text{ in }\ L^{\infty}(\Omega),\quad g_n\underset n\longrightarrow g\quad\text{ in }\ C^2(\overline\Omega), \quad \psi_n\underset n\longrightarrow\psi \quad\text{ in }\ W^{2,\infty}(\Omega),\quad |\nabla\psi|<g.
	\end{equation}
Then, fixing $p\in (d,\infty)$, there exists $(\lambda_n,u_n)$  solution of problem \eqref{lmcf} with data $(f_n,g_n,\psi_n)$ and there exists $(\lambda,u)$, solution of problem \eqref{lmcf} with data $(f,g,\psi)$, such that
	\begin{equation*}
	\lambda_n\underset n\lraup \lambda\quad\text{ in }L^p(\Omega)\text{-weak},\quad		u_n\underset n\longrightarrow u\ \text{ in } C^{1,1-d/p}(\overline\Omega).
	\end{equation*}
\end{theorem}
\begin{proof} Let $(\lambda_n,u_n)$ be a solution of problem \eqref{lmcf} obtained as a limit of a subsequence of $(\widehat k_{\eps,n},u_{\eps,n})_\eps$, where $u_{\eps,n}$ solves the approximating problem \eqref{approx} with data $(f_n,g_n,\psi_n)$ and $\widehat k_{\eps,n}=k_\eps(|\nabla u_{\eps,n}|^2-g_n^2)$. 
	
	In the proof of Theorem \ref{+}, which was done in Section 3, we saw that $u_n$ solves the variational inequality \eqref{gradp} with data $(f_n+(-\Delta\psi_n-f_n)^+\chi_{_{\{u_n=\psi_n\}}},g_n)$ and $u_n\in W^{2,p}(\Omega)$ with $\|u_n\|_{W^{2,p}(\Omega)}$ depending only on $\|f_n\|_{L^p(\Omega)}$, $\|g_n\|_{C^2(\overline\Omega)}$ and $\|\Delta\psi_n\|_{L^p(\Omega)}$.  By the assumption \eqref{limn}, we obtain that $\|u_n\|_{W^{2,p}(\Omega)}$ is bounded independently of $n$ and thus, for a subsequence, we have the convergences
	$$u_n\underset n\lraup u\ \text{ in }W^{2,p}(\Omega)\text{-weak},\quad u_n\underset n\longrightarrow u\ \text{ in }C^{1,1-d/p}(\overline\Omega).$$

We know that, by \eqref{213}, $\|\widehat k_{\eps,n}\|_{L^p(\Omega)}\le C_{n}$, where $C_{n}$ depends only  on  $\|f_n\|_{L^\infty(\Omega)} $, $\|g_n\|_{C^2(\overline\Omega)}$ and $\|\Delta\psi_n\|_{L^\infty(\Omega)} $. By the assumption \eqref{limn}, $C_{n}$ is bounded from above by a constant $C$, independent of $n$. For a subsequence, $\widehat k_{\eps,n}\underset{\eps\rightarrow 0}\lraup\lambda_n$ in $L^{p}(\Omega)$-weak and therefore
\begin{equation*}
	\|\lambda_n\|_{L^p(\Omega)}\le\varliminf_{\eps\rightarrow0}\|\widehat k_{\eps,n}\|_{L^p(\Omega)}\le C.
	\end{equation*}
Then, there exists a subsequence of $(\lambda_n)_n$ such that
\begin{equation*}\lambda_n\underset n\lraup \lambda \quad\text{ in } L^p(\Omega)\text{-weak}.
	\end{equation*}
	
Since $\|\chi_{\{u_n=\psi_n\}}\|_{L^p(\Omega)}\le|\Omega|^{\frac1p}$, there exists $\chi_{_*}\in L^p(\Omega)$ and a subsequence such that
$$\chi_{ _{\{u_n=\psi_n\} }}\underset n\lraup\chi_{_*}\quad\text{ in }L^p(\Omega)\text{-weak}.$$ As $(\lambda_n,u_n)$ solves problem \eqref{lmcf}, for any $\varphi\in H^1_0(\Omega)$ we have
$$\int_\Omega\lambda_n\nabla u_n\cdot\nabla\varphi-\int_\Omega(-\Delta\psi_n-f_n)^+\chi_{ _{\{u_n=\psi_n\} }}\varphi=\int_\Omega f_n\varphi.$$
So, letting $n\rightarrow\infty$, as $\nabla u_n\underset n\longrightarrow\nabla u$ in $C(\overline\Omega)$, we get
\begin{equation*}
\int_\Omega\lambda\nabla u\cdot\nabla\varphi-\int_\Omega(-\Delta\psi-f)^+\chi_{_*}\varphi=\int_\Omega f\varphi.
\end{equation*}
Since $|\nabla u_n|\le g_n$ and $u_n\ge\psi_n$, then $|\nabla u|\le g$ and $u\ge\psi$, i.e. $u\in\K_{g,\psi}^\nabla$. As $(\lambda_n-1)(|\nabla u_n|-g_n)=0$, the weak convergence of $\lambda_n$ to $\lambda$ in $L^p(\Omega)$ and the strong convergence of $\nabla u_n$ to $\nabla u$ in $C(\overline\Omega)$ are sufficient to conclude that $(\lambda-1)(|\nabla u|-g)=0$.

In order to prove that $(\lambda,u)$ solves problem \eqref{lmcf}, it remains to verify that 
\begin{equation*}
	(-\Delta\psi-f)^+\chi_{_*}=(-\Delta\psi-f)^+\chi_{ _{\{u=\psi\} }}.
	\end{equation*}
 Although the details are slightly different, we follow the ideas in the proof of Theorem\ref{+}. Notice first that, as $\chi_{ _{\{u_n=\psi_n\} }}(u_n-\psi_n)\equiv0$, then $\chi_{_*}(u-\psi)\equiv0$, which implies that $\chi_{_*}=0$ in $\{u>\psi\}$. 

We know that $u$ solves the variational inequality \eqref{gradp} with data $(f+(-\Delta\psi-f)^+,g)$. 
Arguing as in Step 3 of the proof of Theorem \ref{+}, as $u\in C^1(\overline\Omega)$, the set $\mathcal U=\{|\nabla u|<g\}$ is open. Taking  $\varphi\in\mathcal D(\mathcal U)$, there exists $\delta>0$ such that $u\pm\delta\varphi\in\K_g^\nabla$ and therefore, using these as test functions in this variational inequality, we obtain, after integrating by parts,
$$\pm\delta\int_\Omega(-\Delta u- f-(-\Delta\psi-f)^+\chi_{_*}\varphi\ge0,$$
and so
\begin{equation}\label{finally}-\Delta u-f=(-\Delta\psi-f)^+\chi_{_*}\quad\text{ in }\mathcal U.
	\end{equation}
In particular, as $(-\Delta\psi-f)^+\chi_{_*}\ge 0$ then also $-\Delta u-f\ge0$.

By Lemma \ref{rod}, as $\Delta u,\Delta\psi\in L^p(\Omega)$ we have $\Delta u=\Delta\psi$ in $\{u=\psi\}$ and, finally, using \eqref{finally}, in $\{u=\psi\}$,
$$(-\Delta \psi-f)^+\chi_{_*}=-\Delta u-f=(-\Delta u-f)^+=(-\Delta\psi-f)^+=(-\Delta\psi-f)^+\chi_{_{\{u=\psi\}}},$$
concluding that $(\lambda,u)$ solves \eqref{lmcf}. As $(\lambda,u)$ solves problem \eqref{lmcf}, by applying Remark \ref{lmcf_vi}, then $u$ solves the variational inequality \eqref{vi}.
\end{proof}

\begin{remark}
	We could have proved that the family of convex sets $\K_{g_n,\psi_n}$ converges, in the sense of Mosco, to $\K_{g,\psi}^\nabla$ and used a well known result of Mosco (see \cite{Mosco1969} or \cite[Theorem 4.1, p. 99]{Rodrigues1987}) to deduce the convergence of the sequence of solutions $(u_n)_n$ of the variational inequality \eqref{vi}, with data $(f_n,g_n,\psi_n)_n$ to the solution $u$ of the variational inequality \eqref{vi}, with a sequence of data $(f,g,\psi)$. However, this convergence would only be proved in $H^1_0(\Omega)$, the space where the convex sets are contained. The last theorem shows that, with our assumptions, $u_n$ converges to $u$ in $C^{1,1-d/p}(\overline\Omega)$.
\end{remark}

	\section*{Acknowledgements}
	
\noindent The authors  were partially financed by Portuguese Funds
through FCT (Funda\c c\~ao para a Ci\^encia e a Tecnologia) within the Projects UIDB/00013/2020 and UIDP/00013/2020.

\end{document}